\documentclass[a4paper,10pt]{amsart}

\usepackage[utf8]{inputenc}

\usepackage{graphicx}

\usepackage{cite}

\usepackage{hyperref}

\hypersetup{%
pdftitle={},
pdfsubject={Mathematics},
pdfauthor={Manfred Madritsch},
pdfkeywords={}
hyperindex=true,plainpages=false}

\usepackage{a4wide}

\usepackage{amsmath}
\usepackage{amsfonts}
\usepackage{amssymb}
\usepackage{amsthm}

\newtheorem{lem}{Lemma}[section]
\newtheorem{thm}[lem]{Theorem}

\newtheorem{cor}[lem]{Corollary}

\numberwithin{equation}{section}

\newtheorem*{cor*}{Corollary}
\newtheorem*{thm*}{Theorem}

\theoremstyle{definition}

\newtheorem{example}{Example}

\theoremstyle{remark}
\newtheorem{rem}[lem]{Remark}

\allowdisplaybreaks[3]

\newcommand{\N}{\mathbb{N}}
\newcommand{\Z}{\mathbb{Z}}

\newcommand{\R}{\mathbb{R}}

\renewcommand{\lvert}{\left\vert}
\renewcommand{\rvert}{\right\vert}
\renewcommand{\lVert}{\left\Vert}
\renewcommand{\rVert}{\right\Vert}

\title{Non-normal numbers in dynamical systems fulfilling the specification property}

\author[M. G. Madritsch]{Manfred G. Madritsch}
\address[M. G. Madritsch]{
\noindent 1. Universit\'e de Lorraine, Institut Elie Cartan de Lorraine, UMR 7502, Vandoeuvre-l\`es-Nancy, F-54506, France;\newline
\noindent 2. CNRS, Institut Elie Cartan de Lorraine, UMR 7502, Vandoeuvre-l\`es-Nancy, F-54506, France}
\email{manfred.madritsch@univ-lorraine.fr}

\author[I. Petrykiewicz]{Izabela Petrykiewicz}
\address[I. Petrykiewicz]{Universit\'e Joseph Fourier, Institut Fourier, 100 rue des maths, 38402 St Martin d'H\`{e}res, France
}
\email{izabela.petrykiewicz@ujf-grenoble.fr}

\subjclass[2010]{}


\date{\today}

\begin{document}

\begin{abstract}
  In the present paper we want to focus on this dichotomy of the
  non-normal numbers -- on the one hand they are a set of measure zero
  and on the other hand they are residual -- for dynamical system
  fulfilling the specification property. These dynamical systems are
  motivated by $\beta$-expansions. We consider the limiting
  frequencies of digits in the words of the languagse arising from these
  dynamical systems, and show that not only a typical $x$ in the sense
  of Baire is non-normal, but also its Ces\`aro variants diverge.
\end{abstract}

\maketitle

\section{Introduction}\label{sec:introduction}
Let $N\geq2$ be an integer, called the base, and
$\Sigma:=\{0,1,\ldots,N-1\}$, called the set of digits. Then for every
$x\in[0,1)$ we denote by 
\[
  x=\sum_{h=1}^\infty d_h(x)N^{-h},
\]
where $d_h(x)\in\Sigma$ for all $h\geq1$, the unique non-terminating
$N$-ary expansion of $x$. For every positive integer $n$ and a block of digits
$\mathbf{b}=b_1\ldots b_k\in\Sigma^k$ we write
\[
  \Pi(x,\mathbf{b},n):=\frac{\lvert\{0\leq
    i<n:d_{i+1}(x)=b_1,\ldots,d_{i+k}(x)=b_k\}\rvert}n 
\]
for the frequency of the block $\mathbf{b}$ among the first $n$ digits of the
$N$-ary expansion of $x$. Furthermore, let
\[
  \Pi_k(x,n):=(\Pi(x,\mathbf{b},n))_{\mathbf{b}\in\Sigma^k}
\]
be the vector of frequencies of all blocks $\mathbf{b}$ of length $k$.

Now we call a number $k$-normal if for every block
$\mathbf{b}\in\Sigma^k$ of digits of length $k$, the limit of the frequency
$\Pi(x,\mathbf{b},n)$ exists and equals $N^{-k}$.
A number is called normal with respect to base $N$ if it is $k$-normal for all
$k\geq1$. Furthermore, a number is called absolutely normal if it is normal to any
base $N\geq2$.

On the one hand, it is a classical result due to Borel
\cite{Borel1909:les_probabilites_denombrables} that Lebesgue almost all numbers
are absolutely normal. So the set of normal numbers is large from a measure
theoretical viewpoint.

On the other hand, it suffices for a number to be not normal if the
limit of the frequency vector is not the uniform one.
First results concerning the Hausdorff dimension or the Baire category
of non-normal numbers were obtained by {\v{S}}al{\'a}t
\cite{salat1966:remark_on_normal} and Volkmann \cite{volkmann1964:non_normal_numbers}. 
Stronger variants of non-normal numbers were of recent interest. In particular, Albeverio \textit{et
  al.} \cite{albeverio_pratsiovytyi_torbin2005:singular_probability_distributions, 
albeverio_pratsiovytyi_torbin2005:topological_and_fractal}
considered the fractal structure of essentially non-normal
numbers and their variants. The theory of multifractal divergence
points lead to the investigation of extremely non-normal numbers by Olsen \cite{
olsen2004:applications_multifractal_divergence01,
olsen2004:applications_multifractal_divergence02}
and Olsen and Winter \cite{olsen_winter2003:normal_and_non}. The
important result for our considerations is that both essentially and
extremely non-normal numbers are large from a topological point of
view.
\section{Definitions and statement of result}
We start with the definition of a dynamical system.
Let $M$ be a compact metric space and $\phi:M\to M$ be a continuous map. Then
we call the pair $(M,\phi)$ a (topological) dynamical system.

The second ingredient is the definition of a topological partition.
Let $M$ be a metric space and let $\mathcal{P}=\{P_0,\ldots,P_{N-1}\}$ be a finite
collection of disjoint open sets. Then we call $\mathcal{P}$ a topological
partition (of $M$) if $M$ is the union of the closures $\overline{P_i}$ for
$i=0,\ldots,N-1$, \textit{i.e.}
\[
  M=\overline{P_0}\cup\cdots\cup\overline{P_{N-1}}.
\]

Suppose now that a dynamical system $(M,\phi)$ and a topological partition
$\mathcal{P}=\{P_0,\ldots,P_{N-1}\}$ of $M$ are given. We want to
consider the symbolic dynamical system behind. Therefore, let
$\Sigma=\{0,\ldots,N-1\}$ be the alphabet corresponding to the
topological partition $\mathcal{P}$. Furthermore, define
\[
\Sigma^k=\{0,\ldots,N-1\}^k,\quad
\Sigma^*=\bigcup_{k\geq1}\Sigma^k\cup\{\epsilon\},\quad\text{and}\quad
\Sigma^\N=\{0,\ldots,N-1\}^\N
\]
to be the set of words of length $k$, the set of finite and the set of
infinite words over $\Sigma$, respectively, where $\epsilon$ is the empty word.
For an infinite word $\omega=a_1a_2a_3\ldots\in\Sigma^\N$ and a
positive integer $n$, let $\omega\vert n=a_1a_2\ldots a_n$ denote the
truncation of $\omega$ to the $n$-th place. Finally, for
$\omega\in\Sigma^*$ we denote by $[\omega]$ the cylinder set of all 
infinite words starting with the same letters as $\omega$,
\textit{i.e.}
\[
[\omega]:=\{\gamma\in\Sigma^\N:\gamma\vert\lvert\omega\rvert=\omega\}.
\]


Now we want to describe the shift space that is generated by our
topological partition. Therefore, we
call $\omega=a_1a_2\ldots a_n\in\Sigma^n$ allowed for
$(\mathcal{P},\phi)$ if
\[
  \bigcap_{k=1}^n\phi^{-k}\left(P_{a_k}\right)\neq\emptyset.
\]
Let $\mathcal{L}_{\mathcal{P},\phi}$ be the set of allowed words. Then
$\mathcal{L}_{\mathcal{P},\phi}$ is a language and there is a unique shift
space $X_{\mathcal{P},\phi}\subseteq\Sigma^\N$, whose language is
$\mathcal{L}_{\mathcal{P},\phi}$. We call $X_{\mathcal{P},\phi}\subseteq\Sigma^\N  $ the one-sided
symbolic dynamical system corresponding to
$(\mathcal{P},\phi)$.

Furthermore, we split the language up
corresponding to the length of the words. For $k\geq1$ we denote by
$$\mathcal{L}_k=\{\omega\in\mathcal{L}_{\mathcal{P},\phi}\colon\lvert\omega\rvert=k\}.$$
Then we have that $\mathcal{L}_{\mathcal{P},\phi}=\bigcup_{k=1}^\infty\mathcal{L}_k$.

Finally, for each $\omega=a_1a_2a_3\ldots
\in X_{\mathcal{P},\phi}$ and $n\geq0$ we denote by $D_n(\omega)$ the
cylinder set of order $n$ corresponding to $\omega$ in $M$,
\textit{i.e.}
\[
D_n(\omega):=\bigcap_{k=0}^n\phi^{-k}(P_{a_k})\subseteq M.
\]

After providing all the ingredients necessary for the statement of our
result we want to link this concept with the $N$-ary representations of Section \ref{sec:introduction}.
\begin{example}\label{example:Nary}
Let $M=\R/\Z$ be the circle and $\phi:M\to M$ be defined by $\phi(x)=N x\pmod
1$. We divide $M$ into $N$ subintervals $P_0,\ldots,P_{N-1}$ of the form
$P_i=(i/N,(i+1)/N)$ and let $\Sigma=\{0,\ldots,N-1\}$. Then the
underlying system is the $N$-ary representation. Furthermore, it is easy to
verify that the language $\mathcal{L}_{\mathcal{P},\phi}(x)$ is the set of all
words over $\Sigma$, so that the one-sided symbolic dynamical system
$X_{\mathcal{P},\phi}$ is the full one-sided $N$-shift $\Sigma^\N$.
\end{example}

Our second example will be the main motivation for this paper. In
particular, we will consider $\beta$-expansions, where $\beta>1$ is
not necessarily an integer. These systems are of special interest,
since the underlying symbolic dynamical system is not the
full-shift. The first authors investigating these number systems were
Parry \cite{parry1960:eta_expansions_real} and Renyi
\cite{renyi1957:representations_real_numbers}. For a more modern
account on these number systems we refer the interested reader to the
book of Dajani and Kraaikamp
\cite{dajani_kraaikamp2002:ergodic_theory_numbers}.

\begin{example}\label{example:beta_expansion}
Let $\beta>1$ be a real number and $\phi\colon[0,1)\to[0,1)$ be the transformation given by
$$\phi(x)=\beta x\mod 1.$$
The sets $$P_i:=\left(\frac
  i\beta,\frac{i+1}\beta\right)\quad(i=0,\ldots,\lfloor\beta\rfloor-1)$$
and $$P_{\lfloor\beta\rfloor+1}:=\left(\frac{\lfloor\beta\rfloor}\beta,1\right)$$
together with $\phi$ form a number system partition of $M$. The corresponding
language is called the $\beta$-shift (\textit{cf.}
\cite{renyi1957:representations_real_numbers,
  parry1960:eta_expansions_real, dajani_kraaikamp2002:ergodic_theory_numbers}).
\end{example}

Before extending the notions of normal and non-normal numbers we want
to investigate the properties of the $\beta$-shift in more detail. We
say that a language $\mathcal{L}$ fulfills the specification property
if there exists a positive integer $j\geq0$ such that we can
concatenate any two words $\mathbf{a}$ and $\mathbf{b}$ by padding a
word of length less than $j$ in between, \textit{i.e.} if, for every
pair $\mathbf{a},\mathbf{b}\in\mathcal{L}$, there exists a word
$\mathbf{u}\in \mathcal{L}$ with $\lvert \mathbf{u}\rvert\leq j$ such
that $\mathbf{a}\mathbf{u}\mathbf{b}\in \mathcal{L}$. Furthermore, we
call the language connected of order $j$ if this padding word can
always be chosen of length $j$. Note that the $\beta$-shift fulfills
this property.

Suppose for the rest of the paper that $(M,\phi)$ is a
number system partition, together with a dynamical system
$X_{\mathcal{P},\phi}$ that fulfills the specification
property with a parameter $j$. Since the partition $\mathcal{P}$ and
the transformation $\phi$ are fixed, we may write
$X=X_{\mathcal{P},\phi}$ and
$\mathcal{L}=\mathcal{L}_{\mathcal{P},\phi}$ for short.

In order to extend the definition of normal and thus non-normal
numbers to $M$ we need that the expansion is unique. Therefore, we
suppose that $\bigcap_{n=0}^\infty\overline{D_n(\omega)}$ consists of
exactly one point. This motivates the definition of the map
$\pi_{\mathcal{P},\phi}:X\to M$ by
\[
\{\pi_{\mathcal{P},\phi}(\omega)\}=\bigcap_{n=0}^\infty\overline{D_n(\omega)}.
\]
However, the converse need not be true. In particular, we consider
Example \ref{example:beta_expansion} with $\beta=\frac{1+\sqrt{5}}2$
(the golden mean). Then clearly $\beta^2-\beta-1=0$. Now on the one hand,
every word in $X$ is mapped to a unique real number. However, if we
consider expansion of $\frac1\beta$, which lies between the two
intervals $P_0$ and $P_1$, then,
since $$\frac1\beta=\frac1{1-\frac1{\beta^2}},$$ we get that
$010101\ldots$ and $100000\ldots$ are possible expansions of
$\frac1\beta$. Similarly we get that $101010\ldots$ and $010000\ldots$
are possible expansion of $\frac1{\beta^2}$. 

However, one observes that these ambiguities originate from the
intersections of two partitions $\overline{P_i}\cap\overline{P_j}$ for
$i\neq j$. Thus we concentrate on the inner points, which somehow
correspond to the irrational numbers in the above case of the decimal
expansion. Let
\[
U=\bigcup_{i=0}^{N-1}P_i,
\]
which is an open and dense ($\overline{U}=M$) set. Then for each
$n\geq1$ the set
\[
U_n=\bigcap_{k=0}^{n-1}\phi^{-k}(U),
\]
is open and dense in $M$. Thus by the Baire Category Theorem, the set
\begin{gather}\label{mani:U-infty}
U_\infty=\bigcap_{n=0}^\infty U_n
\end{gather}
is dense. Since $M\setminus U_\infty$ is the countable union of
nowhere dense sets it suffices to show that a set is residual in
$U_\infty$ in order to show that it is in fact residual in $M$. Furthermore,
for $x\in U_\infty$ we may call $\omega$ \textbf{the} symbolic
expansion of $x$ if $\pi_{\mathcal{P},\phi}(\omega)=x$. Thus in the
following we will silently suppose that $x\in U_\infty$.

After defining the environment we want to pull over the definitions of
normal and non-normal numbers to the symbolic dynamical system. To this end let
$\mathbf{b}\in\Sigma^k$ be a block of letters of length $k$ and
$\omega=a_1a_2a_3\ldots\in X$ be the symbolic
representation of an element. Then we write
\[
\mathrm{P}(\omega,\mathbf{b},n)=\frac{\lvert\left\{0\leq i<n:a_{i+1}=b_1,\ldots,a_{i+k}=b_k\right\}\rvert}{n}
\]
for the frequency of the block $\mathbf{b}$ among the first $n$ letters of
$\omega$. In the same manner as above let 
\[
\mathrm{P}_k(\omega,n)=\left(\mathrm{P}(\omega,\mathbf{b},n)\right)_{\mathbf{b}\in\Sigma^k}
\]
be the vector of all frequencies of blocks $\mathbf{b}$ of length $k$ among the
first $n$ letters of $\omega$.

Let $\mu$ be a given $\phi$-invariant probability measure on
$X$ and $\omega\in X$. Then we
call the measure $\mu$ associated to $\omega$ if there exists a
infinite sub-sequence $F$ of $\N$ such that for any block
$\mathbf{b}\in\Sigma^k$
$$\lim_{\substack{n\to\infty\\ n\in
    F}}P(\omega,\mathbf{b},n)=\mu([\mathbf{b}]).$$ Furthermore, we call
$\omega$ a generic point for $\mu$ if we can take $F=\N$: then $\mu$
is the only measure associated with $\omega$. If $\mu$ is the maximal
measure, then we call $\omega$ normal. Finally, for a $\phi$-invariant
probability measure on $X$ we define its entropy by
$$H(\mu)=\lim_{N\to\infty}-\frac1N\sum_{a_1,\ldots,a_n\in A^n}
\mu([a_1,\ldots,a_n])\log(\mu([a_1,\ldots,a_n])).$$

The existence of such an invariant measure for the $\beta$-shift was
independently proven by
Gelfond~\cite{gelprimefond1959:common_property_number} and Parry
\cite{parry1960:eta_expansions_real}. Bertrand-Mathis
\cite{bertrand-mathis1988:points_generiques_de} constructed such an
invariant measure by generalizing the construction of Champernowne for
any dynamical system fulfilling the specification property. She also
showed that this measure is ergodic, strongly mixing, its entropy is
$\log\beta$ and it is generic for the maximal measure. An application
of Birkhoff's ergodic theorem yields that almost all numbers
$\omega\in X$ are normal (\textit{cf.} Chapter 3.1.2 of
\cite{dajani_kraaikamp2002:ergodic_theory_numbers}).

Normal sequences for $\beta$-shifts where constructed by Ito and Shiokawa
\cite{ito_shiokawa1975:construction_eta_normal}, however, these
expansions provide no admissible numbers. Furthermore, Bertrand-Mathis and
Volkmann~\cite{bertrand-mathis_volkmann1989:epsilon_k_normal}
constructed normal numbers on connected dynamical systems.

We note that we equivalently could have defined the measure-theoretic
dynamical system with respect to $M$ instead of
$X$. However, since the definition of essentially
and extremely non-normal numbers does not depend on this, we will not consider this
in the following.

As already mentioned above, the aim of the present paper is to show that the
non-normal numbers are a large set in the topological sense. Sigmund
\cite{sigmund1974:dynamical_systems_with} showed, that for any
dynamical system fulfilling the specification property, the set of
non-normal numbers is residual. However, in the present paper we want
to show that even smaller sets, namely the essentially and extremely
non-normal numbers, are also residual.

We start by defining the simplex
of all probability vectors $\Delta_k$ by
\[
  \Delta_k=\left\{(p_{\mathbf{i}})_{\mathbf{i}\in\mathcal{L}_k}:p_{\mathbf{i}}\geq0,
  \sum_{\mathbf{i}\in\mathcal{L}_k}p_{\mathbf{i}}=1\right\}.
\]

Let $\lVert\cdot\rVert_1$ denote the $1$-norm then
$(\Delta_k,\lVert\cdot\rVert_1)$ is a metric space. On the one hand, we
clearly have that any vector $\mathrm{P}_k(\omega,n)$ of frequencies of
blocks of digits of length $k$ belongs to $\Delta_k$. On the other hand, if we assume for example that the word $11$ is forbidden in the expansion. Then the maximum
frequency for the single letter $0$ is $1$ and for $1$ is
$\frac12$. Therefore, the probability vector $(0,1)$ cannot be reached.

Let $\mathrm{A}_k(\omega)$ be the set of accumulation points of
the sequence $(\mathrm{P}_k(\omega,n))_n$ with respect to
$\lVert\cdot\rVert_1$, \textit{i.e.} for $\omega\in X$ we set
\[
\mathrm{A}_k(\omega):=\left\{\mathbf{p}\in\Delta_k:\mathbf{p}\text{ is an
    accumulation point of }(\mathrm{P}_k(\omega,n))_n\right\}.
\]
Then we define $\mathrm{S}_k$ as union of all possible accumulation points, \textit{i.e.}
$$\mathrm{S}_k:=\bigcup_{\omega\in X}\mathrm{A}_k(\omega).$$
We note that in the case of $N$-ary expansions this definition leads
to the shift invariant probability vectors (\textit{cf.} Theorem 0 of
Olsen \cite{olsen2004:extremely_non_normal}).

We call a number $\omega\in X$ essentially non-normal if
for all $i\in\Sigma$ the limit
\[
  \lim_{n\to\infty}\mathrm{P}(\omega,i,n)
\]
does not exist. For the case of $N$-ary expansions Albeverio
\textit{et
  al.}\cite{albeverio_pratsiovytyi_torbin2005:singular_probability_distributions,
  albeverio_pratsiovytyi_torbin2005:topological_and_fractal} could
prove the following
\begin{thm*}[{\cite[Theorem
    1]{albeverio_pratsiovytyi_torbin2005:singular_probability_distributions,
      albeverio_pratsiovytyi_torbin2005:topological_and_fractal}}]~
  Let $(\mathcal{P},\phi)$ be the $N$-ary representation of Example
  \ref{example:Nary}. Then the set of essentially non-normal numbers is residual.
\end{thm*}

This result has been generalized to Markov partitions whose underlying
language is the full shift by the first author \cite{madritsch2014:non_normal_numbers}. 
Our first results is the following generalization.
\begin{thm}\label{thm:essentially_non_normal}
  Let $\mathcal{P}=\{P_0,\ldots,P_{N-1}\}$ be a number system
  partition for $(M,\phi)$. Suppose that
  \begin{itemize}
  \item $\bigcap_{n=0}^\infty\overline{D_n(\omega)}$ consists of exactly one
  point;
  \item $X_{\mathcal{P},\phi}$ fulfills the specification property;
  \item for all $i\in\Sigma$ there exist 
    $\mathbf{q}_{i,1}=(q_{1,1},...q_{1,N-1}),\mathbf{q}_{i,2}=(q_{2,1},...q_{2,N-1}) \in S_1$
 such that $|q_{1,i}-q_{2,i}|>0$.
  \end{itemize}
  Then the set of essentially
  non-normal numbers is residual.
\end{thm}

\begin{rem}
  The requirement that for each digit we need at least two possible distributions 
  is sufficient in order to prevent that the underlying language is too simple. For
  example, we want to exclude the case of the shift over the alphabet
  $\{0,1\}$ with forbidden words $00$ and $11$. 
\end{rem}



A different concept of non-normal numbers are those being arbitrarily
close to any given configuration. In particular, we want to generalize
the idea of extremely non-normal numbers and their Ces\`aro variants
to the setting of number system partitions. 

For any infinite word $\omega\in X$ we clearly have
$\mathrm{A}_k(\omega)\subset S_k$. On the other hand, we call
$\omega\in X$ extremely non-k-normal if the set of accumulation points
of the sequence $(\mathrm{P}_k(\omega,n))_n$ (with respect to
$\lVert\cdot\rVert_1$) equals $\mathrm{S}_k$, \textit{i.e.}
$\mathrm{A}_k(\omega)=\mathrm{S}_k$.  Furthermore, we call a number
extremly non-normal if it is extremely non-k-normal for all $k\geq1$.

The set of extremely non-normal numbers for the $N$-ary represenation
has been considered by Olsen \cite{olsen2004:extremely_non_normal}.
\begin{thm*}[{\cite[Theorem 1]{olsen2004:extremely_non_normal}}]~
Let $(\mathcal{P},\phi)$ be the $N$-ary expansion of Example
\ref{example:Nary}. Then the set of extremely non-normal numbers is
residual in $M$.
\end{thm*}


This result was generalized to iterated function systems by Baek and
Olsen~\cite{baek_olsen2010:baire_category_and} and to finite Markov
partitions by the first author
\cite{madritsch2014:non_normal_numbers}. Furthermore, number systems
with infinite set of digits like the continued fraction expansion or
L\"uroth expansion were considered by
Olsen~\cite{olsen2003:extremely_non_normal}, {\v
  S}al{\'a}t~\cite{salat1968:zur_metrischen_theorie}, respectively.
Finally, {\v S}al{\'a}t \cite{salat1968:uber_die_cantorschen}
considered the Hausdorff dimension of sets with digital restrictions
with respect to the Cantor series expansion.

We want to extend this notion to the Ces\`aro averages of the
frequencies. To this end for a fixed block $b_1\ldots b_k\in\Sigma^k$ let
\[\mathrm{P}^{(0)}(\omega,\mathbf{b},n)=\mathrm{P}(\omega,\mathbf{b},n).\]
For $r\geq1$ we recursively define
\[
\mathrm{P}^{(r)}(\omega,\mathbf{b},n)=\frac{\sum_{j=1}^n\mathrm{P}^{(r-1)}(\omega,\mathbf{b},j)}{n}
\]
to be the $r$th iterated Ces\`aro average of the frequency of the
block of digits $\mathbf{b}$ under the first $n$ digits. Furthermore,
we define by
\[
\mathrm{P}_k^{(r)}(\omega,n):=\left(\mathrm{P}^{(r)}(\omega,\mathbf{b},n)\right)_{\mathbf{b}\in\Sigma^k}
\]
the vector of $r$th iterated Ces\`aro averages. As above, we are
interested in the accumulation points. Thus similar to above let
$\mathrm{A}^{(r)}_k(\omega)$ denote the set of accumulation points of the sequence
$(\mathrm{P}^{(r)}_k(\omega,n))_n$ with respect to
$\lVert\cdot\rVert_1$,\textit{i.e.}
\[
\mathrm{A}^{(r)}_k(\omega):=\left\{\mathbf{p}\in\Delta_k:\mathbf{p}\text{ is an
    accumulation point of }(\mathrm{P}^{(r)}_k(\omega,n))_n\right\}.
\]
Now we call a number $r$th iterated Ces\`aro extremely non-k-normal if
the set of accumulation points is the full set, \textit{i.e.}
$\mathrm{A}^{(r)}_k=\mathrm{S}_k$. 

For $r\geq1$ and $k\geq1$ we denote by $\mathbb{E}^{(r)}_k$ the set of
$r$th iterated Ces\`aro extremely non-k-normal numbers of $M$.
Furthermore, for $r\geq1$ we denote by $\mathbb{E}^{(r)}$ the set of
$r$th iterated Ces\`aro extremely non-normal numbers and by
$\mathbb{E}$ the set of completely Ces\`aro extremely non-normal
numbers, \textit{i.e.}
\[
\mathbb{E}=\bigcap_{k}\mathbb{E}_k^{(r)}
\quad\text{and}\quad
\mathbb{E}=\bigcap_{r}\mathbb{E}^{(r)}=\bigcap_{r,k}\mathbb{E}_k^{(r)}.
\]

As above, this has already been considered for the case of the $N$-ary
expansion by Hyde \textit{et al.} \cite{hyde10:_iterat_cesar_baire}. 
\begin{thm*}[{\cite[Theorem 1.1]{hyde10:_iterat_cesar_baire}}]
Let $(\mathcal{P},\phi)$ be the $N$-ary representation of Example
\ref{example:Nary}. Then for all $r\geq1$ the set $\mathbb{E}^{(r)}_1$
is residual.
\end{thm*}

In the context of extremely
non-normal numbers our result is the following.
\begin{thm}\label{thm:extremely_non_normal}
Let $k$, $r$ and $N$ be positive integers. Furthermore, let
$\mathcal{P}=\{P_0,\ldots,P_{N-1}\}$ be a number system partition
for $(M,\phi)$. Suppose that $\mathcal{L}_{\mathcal{P},\phi}$ fulfills
the specification property. Then the set $\mathbb{E}^{(r)}_k$ is residual.
\end{thm}

Since the set of non-normal numbers is a countable intersection of
sets $\mathbb{E}^{(r)}_k$ we get the following
\begin{cor}
  Let $N$ be a positive integer and
  $\mathcal{P}=\{P_0,\ldots,P_{N-1}\}$ be a number system partition
  for $(M,\phi)$. Suppose that $\mathcal{L}_{\mathcal{P},\phi}$
  fulfills the specification property. Then the sets
  $\mathbb{E}^{(r)}$ and $\mathbb{E}$ are residual.
\end{cor}

\section{Proof of Theorem \ref{thm:essentially_non_normal}}\label{proof:essentially_non_normal}

Before we start proving Theorem\ref{thm:essentially_non_normal}, we will construct sets $Z_n$ 
which we will use in order to ``measure'' the distance between
the proportion of occurrences of blocks and $\mathbf{q}$.
Let $k\in\N$ and $\mathbf{q}\in S_k$ be
fixed. For $n\geq1$ let
\[
Z_n=Z_n(\mathbf{q},k)=\left\{\omega\in\bigcup_{\ell\geq kn\lvert\mathcal{L}_k\rvert}\mathcal{L}_\ell\vert
\lVert\mathrm{P}_k(\omega)-\mathbf{q}\rVert_1\leq\frac1n\right\}.
\]

The main idea consists now in the construction of a word having the
desired frequencies. In particular, for a given word $\omega$ we want
to show that we can add a word from $Z_n$
to get a word whose frequency vector is sufficiently near to
$\mathbf{q}$. Therefore, we first need that $Z_n$ is not empty.

\begin{thm}\label{thm:equalityspectrum}
Let $\mathbf{q}\in\mathrm{S}_k$. Then
$$\dim\{\omega\in X\colon\lim_{n\to\infty}\mathrm{P}_k(\omega,n)=\mathbf{q}\}=\inf_{\mathbf{q}\in\mathrm{S}_k} H(\mathbf{q}).$$
\end{thm}

\begin{proof}
  This is essentially Theorem 6 of
  \cite{olsen2003:multifractal_analysis_divergence} (see also Theorem
  7.1 of
  \cite{olsen_winter2007:multifractal_analysis_divergence}).
  
  In our considerations we have two main differences. On the one hand,
  we consider dynamical system fulfilling the specification property,
  whereas Olsen \cite{olsen2003:multifractal_analysis_divergence}
  investigates subshifts of finite type modelled by a directed and
  strongly connected multigraph. However, in his proof he never uses
  the finitude of the set of exceptions. This means that they stay
  true if we replace the subshift of finite type by one fulfilling the
  specification property.

  Another difference is that we have a number system partition,
  whereas Olsen \cite{olsen2003:multifractal_analysis_divergence}
  analyses a graph directed self-conformal iterated function system
  satisfying the Strong Open Set Condition. In the iterated function
  system, we have first the functions and the partition and in our
  case we have first a partition and then the restricted
  function. Therefore, this changes only the point of view. Furthermore,
  the Strong Open Set Condition is satisfied by the topological
  partition, which we use.

  After adapting these differences the proof runs along the same lines.
\end{proof}

\begin{lem}\label{lem:Znotempty}
For all $n\geq1$, $\mathbf{q}\in\mathrm{S}_k$ and $k\in\N$
we have $Z_n(\mathbf{q},k)\neq\emptyset$.
\end{lem}

\begin{proof}
If follows from Theorem \ref{thm:equalityspectrum} that
$$\dim\{\omega\in X\colon\lim_{n\to\infty}\mathrm{P}_k(\omega,n)=\mathbf{q}\}=\inf_{\mathbf{q}\in\mathrm{S}_k} H(\mathbf{q})>0,$$
which implies that $\{\omega\in
X\colon\lim_{n\to\infty}\mathrm{P}_k(\omega,n)=\mathbf{q}\}\neq\emptyset$. Thus
we chose a $\omega\in X$ such that
$\lim_{n\to\infty}\mathrm{P}_k(\omega,n)=\mathbf{q}$. Then for
sufficiently large $\ell$ the truncated word $\omega\mid\ell$ lies in
$Z_n$.
\end{proof}

Since we may not put any two words together, we use the specification
property do define a modified concatenation. For any pair of finite
words $\mathbf{a}$ and $\mathbf{b}$ we fix a
$\mathbf{u}_{\mathbf{a},\mathbf{b}}$ with
$\lvert\mathbf{u}_{\mathbf{a},\mathbf{b}}\rvert\leq j$ such that
$\mathbf{a}\mathbf{u}_{\mathbf{a},\mathbf{b}}\mathbf{b}\in\mathcal{L}$. Then
for $\mathbf{a}_1,\ldots,\mathbf{a}_m\in\mathcal{L}$ and $n\in\N$ we
write
$$\mathbf{a}_1\odot\mathbf{a}_2\odot\cdots\odot\mathbf{a}_m
:=\mathbf{a}_1\mathbf{u}_{\mathbf{a}_1,\mathbf{a}_2}\mathbf{a}_2
\mathbf{u}_{\mathbf{a}_2,\mathbf{a}_3}\mathbf{a}_3\cdots
\mathbf{a}_{m-1}\mathbf{u}_{\mathbf{a}_{m-1},\mathbf{a}_m}\mathbf{a}_m$$
and
$$\mathbf{a}^{\odot n}:=
\underbrace{\mathbf{a}\odot\mathbf{a}\odot\cdots\odot\mathbf{a}}_{n\text
 { times}}.$$
Then we have the following result. 

\begin{lem}\label{lem:distributions}
Let $k \in \N$, and let $\mathbf{q}_1, \mathbf{q}_2, ..., \mathbf{q}_m \in S_k$. For every $\varepsilon >0$, and every $\omega_0 \in \mathcal{L}_{\mathcal{P},\phi}$, there exists $\omega \in \mathcal{L}_{\mathcal{P},\phi}$ whose prefix is $\omega_0$, and $n_1< n_2<...<n_m$ such that for all $1\leq i \leq m$ we have
\begin{equation}\label{eq:distributions}
\lVert\mathrm{P}_k(\omega,n_i)-\mathbf{q}_i\rVert_1\leq \varepsilon.
\end{equation}
\end{lem}
\begin{proof}
Let $\varepsilon>0$ and $\omega_0 \in \mathcal{L}_{\mathcal{P},\phi}$ be given. We will define $\omega=\omega_o\odot\omega_1\odot\omega_2\odot...\odot\omega_m$ recursively. 
For $i\geq1$, let $$l_i\geq \frac{2(|\omega_0\odot\omega_1\odot...\odot\omega_{i-1}|+k+j-1)|\mathcal{L}_k|}{\varepsilon}.$$
Choose any $\omega_i \in Z_{l_i}(\mathbf{q}_i,k)$. We now show that (\ref{eq:distributions}) is satisfied with $\omega=\omega_0\odot\omega_1\odot...\odot\omega_{i}$ and $n_i=|\omega_0\odot\omega_1\odot...\odot\omega_{i}|$.\\
Let $\mathbf{q}_i=(q_{\mathbf{b_1}},...,q_{\mathbf{b_{|\mathcal{L}_k|}}})$ We have
\begin{align*}
\lVert\mathrm{P}_k(\omega_0\odot...\odot\omega_{i},n_i)-\mathbf{q}_i\rVert_1&=\sum_{\mathbf{b}\in \mathcal{L}_k} |\mathrm{P}(\omega_0\odot...\odot\omega_{i},\mathbf{b})-q_{\mathbf{b}}|\\
&\leq \sum_{\mathbf{b}\in \mathcal{L}_k} |\mathrm{P}(\omega_0\odot...\odot\omega_{i},\mathbf{b})-\mathrm{P}(\omega_{i},\mathbf{b})|
+ \sum_{\mathbf{b}\in \mathcal{L}_k} |\mathrm{P}(\omega_{i},\mathbf{b})-q_{\mathbf{b}}|\\
&\leq \sum_{\mathbf{b}\in \mathcal{L}_k} |\mathrm{P}(\omega_0\odot...\odot\omega_{i},\mathbf{b})-\mathrm{P}(\omega_{i},\mathbf{b})|
+ \frac{|\mathcal{L}_k|}{l_i}\\
&\leq \sum_{\mathbf{b}\in \mathcal{L}_k} |\mathrm{P}(\omega_0\odot...\odot\omega_{i},\mathbf{b})-\mathrm{P}(\omega_{i},\mathbf{b})|
+ \frac{\varepsilon}{2},
\end{align*}
by Lemma \ref{lem:Znotempty} and our choice of $l_i$. Now consider $|\mathrm{P}(\omega_0\odot...\odot\omega_{i},\mathbf{b})-\mathrm{P}(\omega_{i},\mathbf{b})|$. The block $\mathbf{b}$ can occur in $\omega_0\odot\omega_1\odot...\odot\omega_{i-1}$, $\omega_i$, and between these two words, hence we have
\begin{align*}
|\mathrm{P}(\omega_0\odot...\odot\omega_{i},\mathbf{b})-&\mathrm{P}(\omega_{i},\mathbf{b})|\\
=&\max\left(\mathrm{P}(\omega_0\odot...\odot\omega_{i},\mathbf{b})-\mathrm{P}(\omega_{i},\mathrm{P}(\omega_{i},\mathbf{b})-\mathbf{b}), \mathrm{P}(\omega_0\odot...\odot\omega_{i},\mathbf{b})\right)\\
\leq& \max\bigg(\frac{\mathrm{P}(\omega_0\odot...\odot\omega_{i-1},\mathbf{b})+ \mathrm{P}(\omega_{i},\mathbf{b})+(k+j-1)}{|\omega_0\odot...\odot\omega_{i}|}-\frac{\mathrm{P}(\omega_{i},\mathbf{b})}{|\omega_i|},\\
&\frac{\mathrm{P}(\omega_i,\mathbf{b})}{|\omega_i|}-\frac{\mathrm{P}(\omega_0\odot...\odot\omega_{i-1},\mathbf{b})+ \mathrm{P}(\omega_{i},\mathbf{b})}{|\omega_0\odot...\odot\omega_{i}|}\bigg)\\
\leq& \max\left(\frac{|\omega_0\odot...\odot\omega_{i-1}|+(k+j-1)}{|\omega_{i}|}, \frac{\omega_0\odot...\odot\omega_{i-1}|+j}{|\omega_{i}|}\right)\\
\leq& \frac{|\omega_0\odot...\odot\omega_{i-1}|+(k+j-1)}{|l_{i}|} \leq \frac{\varepsilon}{2|\mathcal{L}_k|},
\end{align*}
which follows by the choice of $l_i$.

This implies that
$$\lVert\mathrm{P}_k(\omega_0\odot...\odot\omega_{i},n_i)-\mathbf{q}_i\rVert_1 \leq \varepsilon$$
completing the proof of the lemma.
\end{proof}

We have now all the ingredients needed to prove Theorem \ref{thm:essentially_non_normal}.

\noindent\textit{Proof of Theorem \ref{thm:essentially_non_normal}.} Let $\mathbf{q}_1,...,\mathbf{q}_m \in S_1$ be
such that for all $i\in\Sigma$ there exists 
$\mathbf{q}_{1,i}=(q_{i,1},...q_{i,N-1}),\mathbf{q}_{2,i}=(q'_{i,1},...q'_{i,N-1})\in \{\mathbf{q}_1,...,\mathbf{q}_m\}$ 
  with $|q_{i,i}-q'_{i,i}|>0$. Let $0<\varepsilon<\frac{\min_{i\in\Sigma}|q_{i,i}-q'_{i,i}|}{2}$. 
  For each $\omega\in\mathcal{L}_{\mathcal{P},\phi}$ let $\omega\odot u_{\omega,\varepsilon}$ be a word described by Lemma \ref{lem:distributions}.
Then for all $n \in \N$, we define sets $C_n$ as follows: 
$$C_n=\{\omega\odot u_{\omega,\varepsilon}\alpha_1 \alpha_2...  \in X_{\mathcal{P},\phi} |  |\omega|=n, \text{ and } \alpha_i \in \Sigma\}.$$
  Let $I_n$ be the interior of $C_n$. Let $D_n=\cup_{k=n}^\infty I_k$, and $F= \cap_{n=1}^{\infty}D_n$. It is clear that $D_n$ is open and dense in $X_{\mathcal{P},\phi}$. Since $F$ is a countable intersection of open and dense sets, it is residual. We now need to show that if $w\in F$, then $w$ is essentially non-normal. 
  Let $w\in F$. Then there exists $(n_k)_k \subseteq \N$ such that $w \in C_{n_k}$. Lemma \ref{lem:distributions} then implies that for each digit of $i$, the sequence $(\mathrm{P}(\omega, i, n))_n$ does not converge. \hfill\qedsymbol

\section{Proof of Theorem \ref{thm:extremely_non_normal}}

Now we draw our attention to the case of extremely non-normal numbers
and their Ces\`aro variants. Let $k\in\N$ and $\mathbf{q}\in S_k$ be
fixed throughout the rest of this section. 
We consider how many copies of elements in $Z_n$ we have to add in
order to get the desired properties.

\begin{lem}\label{lem:frequency:word}
Let $\mathbf{q}\in\mathrm{S}_k$ and $n,t\in\N$ be positive
integers. Furthermore, let $\omega=\omega_1\ldots\omega_t\in\mathcal{L}_t$ be a
word of length $t$. Then, for any $\gamma\in Z_n(\mathbf{q},k)$ and any 
\begin{gather}\label{m:bound}
\ell\geq R:=t\left(1+\frac{\lvert\gamma\rvert}k\right).
\end{gather}
we get that
\[
\lVert\mathrm{P}_k(\omega\odot\gamma^{\odot\ell})-\mathbf{q}\rVert\leq\frac4n.
\]
\end{lem}

\begin{proof}
We set $s:=\lvert\gamma\rvert$,
$\sigma:=\omega\odot\gamma^{\odot\ell}$ and $L:=\lvert\sigma\rvert$.
For a fixed block $\mathbf{i}$ an occurrence can happen in $\omega$, in
$\gamma$ or somewhere in between. Thus for every
$\mathbf{i}\in\Sigma^k$ we clearly have that 
\[
\frac{\ell s}{L}\mathrm{P}(\gamma,\mathbf{i})
\leq \mathrm{P}(\sigma,\mathbf{i})
\leq\frac{\ell s\mathrm{P}(\gamma,\mathbf{i})}{L}+\frac{t+\ell(k+j-1)}{L}.
\]

Now we concentrate on the occurrences inside the copies of $\gamma$
and show that we may neglect all other occurrences, \textit{i.e.}
\[
\lVert\mathrm{P}_k(\sigma)-\mathbf{q}\rVert
\leq\lVert\mathrm{P}_k(\sigma)-\frac{\ell s}L \mathrm{P}_k(\gamma)\rVert+
\lVert\frac{\ell s}L \mathrm{P}_k(\gamma)-\mathbf{q}\rVert.
\]
We will estimate both parts separately. For the first one we get that
\begin{align*}
\lVert\mathrm{P}_k(\sigma)-\frac{\ell s}L
\mathrm{P}(\gamma,\mathbf{i})\rVert
&=\sum_{\mathbf{i}\in\Sigma^k}\lvert\mathrm{P}(\sigma,\mathbf{i})-\frac{\ell
  s}L
\mathrm{P}(\gamma,\mathbf{i})\rvert
\leq\sum_{\mathbf{i}\in\Sigma^k}\frac{t+\ell(k+j)}{L}\\
&\leq \lvert\mathcal{L}_k\rvert\frac{t+\ell(k+j)}{\ell n(k+j)\lvert\mathcal{L}_k\rvert}
=\frac{t}{\ell n(k+j)}+\frac1n.
\end{align*}
where we have used that $L\geq \ell s\geq \ell n(k+j)\lvert\mathcal{L}_k\rvert$.

For the second part we get that
\begin{align*}
\lVert\frac{\ell s}L \mathrm{P}_k(\gamma)-\mathbf{q}\rVert
&\leq\lVert\frac{\ell s}L\mathrm{P}_k(\gamma)-\mathrm{P}_k(\gamma)\rVert
+\lVert\mathrm{P}_k(\gamma)-\mathbf{q}\rVert\\
&\leq\ell s\lvert\frac1L-\frac1{\ell s}\rvert+\frac1n\\
&\leq \frac{t+\ell j}L+\frac1n\leq\frac{t+\ell j}{\ell n(k+j)\lvert\mathcal{L}_k\rvert}+\frac1n.
\end{align*}

Putting these together yields
\[
\lVert\mathrm{P}_k(\sigma)-\mathbf{q}\rVert\leq
\frac2n+\frac{t}{\ell n(k+j)}+\frac{t+\ell j}{\ell n(k+j)\lvert\mathcal{L}_k\rvert}.
\]

By our assumptions on the size of $\ell$ in \eqref{m:bound} this proves the
lemma.
\end{proof}

As in the papers of Olsen
\cite{olsen2004:extremely_non_normal,olsen2003:extremely_non_normal} our main
idea is to construct a residual set $E\subset\mathbb{E}_k^{(r)}$. But
before we start we want to ease up notation. To this end we recursively define the
function $\varphi_1(x)=2^x$ and $\varphi_m(x)=\varphi_1(\varphi_{m-1}(x))$ for
$m\geq2$. Furthermore, we set
$\mathbb{D}=(\mathbb{Q}^N\cap\mathrm{S}_k)$. Since $\mathbb{D}$ is
countable and dense in $\mathrm{S}_k$ we may concentrate on the
probability vectors $\mathbf{q}\in\mathbb{D}$.

Now we say that a sequence $(\mathbf{x}_n)_n$ in $\R^{N^k}$ has property $P$ if
for all $\mathbf{q}\in\mathbb{D}$, $m\in\N$, $i\in\N$, and
$\varepsilon>0$, there exists a $j\in\N$ satisfying:
\begin{enumerate}
\item $j\geq i$,
\item $j/2^j<\varepsilon$,
\item if $j<n<\varphi_m(j)$ then
  $\lVert\mathbf{x}_n-\mathbf{q}\rVert<\varepsilon$.
\end{enumerate}

Then we define our set $E$ to consist of all frequency vectors having property
$P$, \textit{i.e.} 
\[
E=\{x\in U_\infty:\text{ $(\mathrm{P}^{(1)}_k(x;n))_{n=1}^\infty$ has
  property $P$}\}.
\]

We will proceed in three steps showing that
\begin{enumerate}
\item $E$ is residual, 
\item if $(\mathrm{P}^{(r)}(x;n))_{n=1}^\infty$ has property $P$, then
  also $(\mathrm{P}^{(r+1)}(x;n))_{n=1}^\infty$ has property $P$, and
\item $E\subseteq\mathbb{E}^{(r)}_k$.
\end{enumerate}

\begin{lem}\label{lem:Eresidual}
The set $E$ is residual.
\end{lem}

\begin{proof}
For fixed $h,m,i\in\N$ and $\mathbf{q}\in\mathbb{D}$, we say that a
sequence $(\mathbf{x}_n)_n$ in $\R^{N^k}$ has property $P_{h,m,\mathbf{q},i}$ if
for every $\varepsilon>1/h$, there exists $j\in\N$ satisfying:
\begin{enumerate}
\item $j\geq i$,
\item $j/2^j<\varepsilon$,
\item if $j<n<\varphi_m(2^j)$, then $\lVert\mathbf{x}_n-\mathbf{q}\rVert<\varepsilon$.
\end{enumerate}
Now let $E_{h,m,\mathbf{q},i}$ be the set of all points whose
frequency vector satisfies property $P_{h,m,\mathbf{q},i}$,
\textit{i.e.}
\[
E_{h,m,\mathbf{q},i}:=\left\{x\in U_{\infty}:\left(\mathrm{P}^{(1)}_k(x;n)\right)_{n=1}^\infty\text{ has
  property }P_{h,m,\mathbf{q},i}\right\}.
\]
Obviously we have that
\[
E=\bigcap_{h\in\N}\bigcap_{m\in\N}\bigcap_{\mathbf{q}\in\mathbb{D}}\bigcap_{i\in\N}E_{h,m,\mathbf{q},i}.
\]

Thus it remains to show, that $E_{h,m,\mathbf{q},i}$ is open and dense.
\begin{enumerate}
\item\textbf{$E_{h,m,\mathbf{q},i}$ is open.} Let $x\in
  E_{h,m,\mathbf{q},i}$, then there exists a $j\in\N$ such that $j\geq
  i$, $j/2^j<1/h$, and if $j<n<\varphi_m(2^j)$, then
  \[
  \lVert\mathrm{P}^{(1)}_k(x;n)-\mathbf{q}\rVert_1<1/h.
  \]

  Let $\omega\in X$ be such that $x=\pi(\omega)$ and set
  $t:=\varphi_m(2^j)$. Since $D_{t}(\omega)$ is open, there exists a
  $\delta>0$ such that the ball $B(x,\delta)\subseteq
  D_{t}(\omega)$. Furthermore, since all $y\in D_{t}(\omega)$ have
  their first $\varphi_m(2^j)$ digits the same as $x$, we get that
  \[
  B(x,\delta)\subseteq D_{t}(\omega)\subseteq E_{h,m,\mathbf{q},i}.
  \]

\item\textbf{$E_{h,m,\mathbf{q},i}$ is dense.} Let $x\in U_\infty$ and
  $\delta>0$. We must find $y\in B(x,\delta)\cap E_{h,m,\mathbf{q},i}$.
  
  Let $\omega\in X$ be such that $x=\pi(\omega)$. Since
  $\overline{D_t(\omega)}\to0$ and $x\in D_t(\omega)$ there
  exists a $t$ such that $D_t(\omega)\subset B(x,\delta)$. Let
  $\sigma=\omega\vert t$ be the first $t$ digits of $x$.

  Now, an application of Lemma \ref{lem:Znotempty} yields that there exists a finite
  word $\gamma$ such that
  \[
    \lVert\mathrm{P}_k(\gamma)-\mathbf{q}\rVert\leq\frac1{6h}.
  \]
  
  Let $\varepsilon\geq\frac1h$ and $L$ be as in the statement of Lemma
  \ref{lem:frequency:word}. Then we choose $j$ such that
  \[
  \frac{j}{2^j}<\varepsilon\quad\text{and}\quad j\geq \max\left(L,i\right).
  \]
  An application of Lemma \ref{lem:frequency:word} then gives us that
  \[
    \lVert\mathrm{P}_k(\sigma\gamma^*\vert j)-\mathbf{q}\rVert\leq\frac6n=\frac1h\leq\varepsilon.
  \]

  Thus we choose $y\in D_{j}(\sigma\gamma^*)$. Then on the one hand,
  $y\in D_{j}(\sigma\gamma^*)\subset D_{t}(\omega)\subset B(x,\delta)$ and on the other
  hand, $y\in D_{j}(\sigma\gamma^*)\subset E_{h,m,\mathbf{q},i}$

\end{enumerate}
It follows that $E$ is the countable intersection of open and dense
sets and therefore $E$ is residual in $U_\infty$.
\end{proof}

\begin{lem}\label{lem:propertyP}
Let $\omega\in X_{\mathcal{P},\phi}$. If
$(\mathrm{P}^{(r)}(\omega,n))_{n=1}^\infty$ has property $P$, then
also $(\mathrm{P}^{(r+1)}(\omega,n))_{n=1}^\infty$ has property $P$.
\end{lem}

This is Lemma 2.2 of \cite{hyde10:_iterat_cesar_baire}. However, the
proof is short so we present it here for completeness.

\begin{proof}
Let $\omega\in X_{\mathcal{P},\phi}$ be such that
$(\mathrm{P}_k^{(r)}(\omega;n))_{n=1}^\infty$ has property 
$P$, and fix $\varepsilon>0,\mathbf{q}\in\mathrm{S}_k$, $i\in\N$ and
$m\in\N$. Since $(\mathrm{P}_k^{(r)}(\omega,n))_{n=1}^\infty$ has property
$P$, there exists $j'\in\N$ with $j'\geq i$,
$j'/2^{j'}<\varepsilon/3$, and such that for
$j'<n<\varphi_{m+1}(2^{j'})$ we have that
$\lVert\mathrm{P}_k^{(r)}(\omega,n)-\mathbf{q}\rVert<\varepsilon/3$.

We set $j=2^{j'}$ and show that
$(\mathrm{P}_k^{(r+1)}(\omega,n))_{n=1}^\infty$ has property $P$ with this
$j$. For all $j<n<\varphi_m(2^j)$ (\textit{i.e.}
$2^{j'}<n<\varphi_{m+1}(2^{j'})$), we have
\begin{align*}
\lVert\mathrm{P}_k^{(r+1)}(\omega,n)-\mathbf{q}\rVert
&=\lVert\frac{\mathrm{P}_k^{(r)}(\omega,1)+\mathrm{P}_k^{(r)}(\omega,2)+\cdots+\mathrm{P}_k^{(r)}(\omega,n)}{n}-\mathbf{q}\rVert\\
&=\lVert\frac{\mathrm{P}_k^{(r)}(\omega,1)+\mathrm{P}_k^{(r)}(\omega,2)+\cdots+\mathrm{P}_k^{(r)}(\omega,j')}{n}\right.\\
&\quad+\left.\frac{\mathrm{P}_k^{(r)}(\omega,j'+1)+\mathrm{P}_k^{(r)}(\omega,2)+\cdots+\mathrm{P}_k^{(r)}(\omega,n)-(n-j')\mathbf{q}}{n}-\frac{j'\mathbf{q}}{n}\rVert\\
&\leq\frac{\lVert\mathrm{P}_k^{(r)}(\omega,1)+\mathrm{P}_k^{(r)}(\omega,2)+\cdots+\mathrm{P}_k^{(r)}(\omega,j')\rVert}{n}\\
&\quad+\frac{\lVert\mathrm{P}_k^{(r)}(\omega,j'+1)-\mathbf{q}\rVert+\cdots+\lVert\mathrm{P}_k^{(r)}(\omega,n)-\mathbf{q}\rVert}{n}
  -\frac{\lVert j'\mathbf{q}\rVert}{n}\\
&\leq\frac{j'}{n}+\frac{\varepsilon}3\frac{n-j'}{n}+\frac{j'}{n}
\leq\frac{j'}{2^{j'}}+\frac{\varepsilon}3+\frac{j'}{2^{j'}}
\leq\frac{\varepsilon}3+\frac{\varepsilon}3+\frac{\varepsilon}3
=\varepsilon.
\end{align*}
\end{proof}

\begin{lem}\label{lem:Esubset}
The set $E$ is a subset of $\mathbb{E}^{(r)}_k$.
\end{lem}

\begin{proof}
We will show, that for any $x\in E$ we also have
$x\in\mathbb{E}^{(r)}_k$. To this end, let $x\in E$ and $\omega\in
X_{\mathcal{P},\phi}$ be the symbolic expansion of $x$. Since
$(\mathrm{P}_k^{(1)}(\omega,n))_n$ has property $P$, by Lemma \ref{lem:propertyP}
we get that $(\mathrm{P}_k^{(r)}(\omega,n)$ has property $P$.

Thus it suffices to show that $\mathbf{p}$ is an accumulation point of
$(\mathrm{P}_k^{(r)}(\omega,n))_n$ for any $\mathbf{p}\in\mathrm{S}_k$. Therefore,
we fix $h\in\N$ and find a $\mathbf{q}\in\mathbb{D}$ such that
\[
\lVert\mathbf{p}-\mathbf{q}\rVert<\frac1h.
\]
Since $(\mathrm{P}_k^{(r)}(\omega,n))_n$ has property $P$ for any $m\in\N$ we
find $j\in\N$ with $j\geq h$ and such that if $j<n<\varphi_m(2^j)$ then
$\lVert\mathrm{P}_k^{(r)}(\omega,n)-\mathbf{q}\rVert<\frac1h$. Hence let $n_h$
be any integer with $j<n_h<\varphi_m(2^j)$, then
\[
\lVert\mathrm{P}_k^{(r)}(\omega,n_h)-\mathbf{q}\rVert<\frac1h.
\]

Thus each $n_h$ in the sequence $(n_h)_h$ satisfies
\[
\lVert\mathbf{p}-\mathrm{P}_k^{(r)}(\omega,n_h)\rVert
\leq\lVert\mathbf{p}-\mathbf{q}\rVert+\lVert\mathrm{P}_k^{(r)}(\omega,n_h)-\mathbf{q}\rVert<\frac2h.
\]
Since $n_h>h$ we may extract an increasing sub-sequence $(n_{h_u})_u$ such that
$\mathrm{P}_k^{(r)}(\omega,n_{h_u})\rightarrow\mathbf{p}$ for
$u\rightarrow\infty$. Thus $\mathbf{p}$ is an accumulation point of
$\mathrm{P}_k^{(r)}(\omega,n)$, which proves the lemma.
\end{proof}

\begin{proof}[Proof of Theorem  \ref{thm:extremely_non_normal}]
Since by Lemma \ref{lem:Eresidual} $E$ is residual in $U_\infty$ and
by Lemma \ref{lem:Esubset} $E$ is a subset of $\mathbb{E}^{(r)}_k$ we get that
$\mathbb{E}^{(r)}_k$ is residual in $U_\infty$. Again we note that
$M\setminus U_\infty$ is the countable union of nowhere dense sets and
therefore $\mathbb{E}^{(r)}_k$ is also residual in $M$. 
\end{proof}


\providecommand{\bysame}{\leavevmode\hbox to3em{\hrulefill}\thinspace}
\providecommand{\MR}{\relax\ifhmode\unskip\space\fi MR }
\providecommand{\MRhref}[2]{%
  \href{http://www.ams.org/mathscinet-getitem?mr=#1}{#2}
}
\providecommand{\href}[2]{#2}

\end{document}